\documentclass[journal]{IEEEtran}

\usepackage{epsf}
\usepackage{subfigure}
\usepackage{bbm}
\usepackage{dsfont}
\usepackage{algorithm}
\usepackage{algorithmic}
\usepackage{textcomp }
\usepackage{amsthm}
\usepackage{amsmath,amssymb}
\usepackage{graphicx}
\usepackage{lipsum}
\usepackage{caption}
\usepackage{stackrel}
\usepackage{subfigure}
\usepackage{amsmath,amssymb}
\usepackage{graphicx}
\usepackage{epstopdf}
\usepackage{color}
\usepackage{soul}
\PassOptionsToPackage{round}{natbib}
\usepackage{hyperref}

\usepackage{cite}
\usepackage{multirow,tabularx}
\usepackage{ifthen}
\usepackage{times}
\usepackage{array,url,kantlipsum}
\usepackage{amsmath}

\usepackage{multicol}
\usepackage[english]{babel}
\usepackage{blindtext}

\usepackage{lipsum, babel}
\DeclareMathAlphabet{\mathpzc}{OT1}{pzc}{m}{it}
\usepackage{amsmath}
\usepackage{amsfonts}
\usepackage{amssymb}
\usepackage{amsxtra}
\usepackage{amsthm}
\usepackage{latexsym}

\makeatletter
\newcommand{\removelatexerror}{\let\@latex@error\@gobble}
\makeatother
\usepackage{lipsum}
\floatstyle{plain}
\newfloat{twocolequfloat}{b}{zzz}
\floatname{twocolequfloat}{Equation}

\newtheorem{lem}{Lemma}

\newtheorem{theorem}{Theorem}
\newtheorem*{theorem*}{Theorem}

\usepackage{mathtools}

\begin{document}

\title{A Lower Bound To The Expected Discovery Time In One-Way Random Neighbor Discovery
 \thanks{The authors are at the Ming Hsieh Department of Electrical Engineering, University of Southern California, Los Angeles, CA 90089, USA. (e-mail: \{burghal,saberteh,molisch\}@usc.edu)}}

\author{
Daoud Burghal {\it Student Member, IEEE,} Arash Saber Tehrani,  \\ Andreas F. Molisch, { \it Fellow, IEEE}  }

\maketitle

\begin{abstract}
In this work we provide a lower bound on the expected discovery time in a one-way neighbor discovery, in particular we show that the average time that a node $i$ takes to discover all of its neighbors is lower bounded by the reciprocal of the average probabilities of successful discovery.
\end{abstract} 
\section{Introduction} 
In random one-way neighbor discovery scheme, nodes announce their identity (ID) at each time instant $t$ with given probabilities. In multi antenna systems, nodes need to announce their ID in all directions. For a general case where transmission probabilities are different  for different node and different direction, the probability that nodes discover one another might not be equal. Under this setting, we provide a bound for the expected time that a node, $i$, takes to discover its neighbors.
\section{The Lower Bound}
As provided in \cite{8086217}, based on \cite{sheldon2002first} and \cite{flajolet1992}, the expected time that node $i$ needs to discover all its neighbors is given by 
 \begin{align} \label{eq:NDET}
 \bar{T}_i = & \sum_{j\in \mathcal{N}_{i}}\frac{1}{P_{j,i}} - \sum_{k,j \in \mathcal{N}_{i}, k \neq j} \frac{1}{P_{j,i}+ P_{k,i}} +... \\ \nonumber & +  (-1)^{N_i+1}  \frac{1}{ \sum_{j \in \mathcal{N}_{i}} P_{j,i}}
 \end{align}
 where $\mathcal{N}_{i}$ is the set of neighbors of node $i$, $N_i$ is the number of neighbors, i.e., $N_i=|\mathcal{N}_{i}|$ , $P_{j,i}$ is the probability that node $i$ discovers node $j$. Defining ${\rm H}_{N_i}$ as the harmonic number of $N_i$, then Theorem 3 in \cite{8086217} states:
 \begin{theorem} \label{Th:NDLowrBoundM}  The expected time that node $i$ takes to discover all its $N_i$ neighbors, $\bar{T}_i$, in an arbitrary network is lower bounded as follows:
$$\bar{T}_i \geq \frac{{\rm H}_{N_i}}{\bar{P}_i },$$
where $\bar{P}_i \triangleq \frac{1}{N_i} \sum_{j \in \mathcal{N}_i } P_{j,i}$.
\end{theorem}

\label{App:NDAppendixA}
\begin{proof}
The proof has two parts; in the first part we show that replacing any two probabilities in (\ref{eq:NDET}) with their average results in a lower bound of the original expected discovery time. In the second part, we show that iteratively substituting two probabilities with their average converges to the right hand side of Theorem \ref{Th:NDLowrBoundM}.

Let $\mathcal{T}^{(j,k)}_i$ denote the discovery time of node $i$, when two probabilities of success, say $P_{j,i}$ and $P_{k,i}$, are substituted with their arithmetic averages, i.e.,
\begin{align} \label{eq:NDTwoAvg}
P_{j,i} \leftarrow \frac{1}{2}(P_{j,i}+P_{k,i}) {\rm ~ and~} P_{k,i} \leftarrow \frac{1}{2}(P_{j,i}+P_{k,i})
\end{align}
The following theorem shows the impact of such modification.
\begin{theorem}\label{Th:NDCombineF} Replacing any two probabilities of success $P_{j,i}$ and $P_{k,i}$ with their averages results in a lower bound of the original expected discovery time, i.e., we have:
$$\mathbb{E}\big\{\mathcal{T}_i\big\} \geq \mathbb{E}\Big\{\mathcal{T}^{(j,k)}_i\Big\}. $$
\end{theorem}
For clarity of presentation we provide the proof in subsection\ref{App:NDAppendixE2}. As indicated above, the formula for $\mathbb{E}\{\mathcal{T}^{(j,k)}_i\}$ is similar to (\ref{eq:NDET}) with new updated probabilities. Consequently, performing the average iteratively will result in a lower bound on each step, i.e., when $j,k,s,r,v,u \in \mathcal{N}_i$ we have,
\begin{align} \label{eq:NDProbSeq}
 \mathbb{E}\Big\{\mathcal{T}^{(j,k)}_i\Big\} \geq \mathbb{E}\Big\{\mathcal{T}^{(j,k)(s,r)}_i\Big\} \geq \mathbb{E}\Big\{\mathcal{T}^{(j,k)(s,r)(v,u)}_i\Big\} \geq \dots
\end{align}
where $\mathbb{E}\big\{\mathcal{T}^{(j,k)(s,r)}_i\big\}$ and $\mathbb{E}\big\{\mathcal{T}^{(j,k)(s,r)(v,u)}_i\big\} $ are the resultant discovery time after applying (\ref{eq:NDTwoAvg}) replacing the new $P_{s,i}$ and $P_{r,i}$ with their averages, and then substituting the resulting $P_{v,i}$ and $P_{u,i}$ with their averages, respectively. Next, we show that this process will ultimately converge, when we apply (\ref{eq:NDTwoAvg}) iteratively to different, possibly randomly chosen, probability pairs. Specifically, we have the following theorem:
\begin{theorem}\label{Th:NDAvgOfTwo}
 Let $\mathbf{p}$ be a vector of real numbers, i.e., $\mathbf{p}=[p_1,...,p_n]^{T}$, and let  $\bar{p} = \frac{1}{n}\sum_{j=1}^n p_j$. There exists an algorithm that iteratively replaces two elements at a time with their average until $\mathbf{p}$ converges to a vector $\bar{\mathbf{p}}$ with all elements equal to $\bar{p}$.
 \end{theorem}
In subsection\ref{App:NDAppendixE3} we use the properties of doubly stochastic and Markov matrices to show that for large $k$, $\mathbf{W}^k \mathbf{p}\rightarrow \bar{\mathbf{p}}$ is one such algorithm, where $\mathbf{W}$ is a matrix that represents a sequence of double averaging and $k$ is the number of repetitions of such procedure. Theorem \ref{Th:NDLowrBoundM} is a direct consequence of Theorem \ref{Th:NDCombineF} and Theorem \ref{Th:NDAvgOfTwo} .
\end{proof}

\subsection{Proof of Theorem \ref{Th:NDCombineF}}\label{App:NDAppendixE2}
\begin{proof}
We start by defining the following terms and functions. Let $\mathfrak{I}^i_{j,k}$ be the constant equal to the sum of the terms in (\ref{eq:NDET}) that are independent of $P_{j,i}$ and $P_{k,i}$, an example of such terms is $\frac{1}{P_{s,i}+P_{r,i}}$, where $s,r,j,k \in \mathcal{N}_{i}$. Let $\mathfrak{C}^i_{k,j}$ be the constant equal to the sum of the terms that include both of $P_{k,i}$ and $P_{j,i}$, an example of such a term is $\frac{1}{P_{j,i}+P_{k,i}}$. Finally, let $\mathfrak{F}^{i}_{j|k}(x)$ be
\begin{align} \label{eq:NDETfunc}
\mathfrak{F}^i_{j|k}(x) &=   \frac{1}{x} -  \sum_{r \in \mathcal{N}_i \setminus \{ j, k\}} \frac{1}{x+P_{r,i}} -... \\ & \nonumber+  (-1)^{(N_i)} \frac{1}{x + \sum_{r \in \mathcal{N}_i \setminus \{ j, k\}} P_{r,i}}.
\end{align}
The function $\mathfrak{F}^{i}_{j|k}(x)$ contains all the terms of (\ref{eq:NDET}) that include $P_{j,i}$ and \emph{not} $P_{k,i}$, with $P_{j,i}$ represented by $x$. From symmetry considerations, it is easy to verify that
$\mathfrak{F}^{i}_{j|k}(x) = \mathfrak{F}^{i}_{k|j}(x)$. 
Thus, we can rewrite (\ref{eq:NDET}) as:
\begin{align}\label{eq:rewriteET}
\bar{T}_i = \mathfrak{F}^i_{j|k}(P_{j,i}) + \mathfrak{F}^i_{k|j}(P_{k,i}) + \mathfrak{C}^i_{k,j} + \mathfrak{I}^i_{j,k}
\end{align}
Next, we show that the function $\mathfrak{F}^i_{j|k}(x)$ is convex function in $x$. 
\begin{lem}\label{Lm:NDFconvex}
For a given finite set $\mathcal{N}_i$ and values $P_{j,i} \in (0,1]$ $\forall j \in \mathcal{N}_i$ the function $\mathfrak{F}^i_{j|k}(x)$ is convex.
\end{lem}
\begin{proof}
We use the second derivative test. Thus, we need to show that the second derivative of $\mathfrak{F}^i_{j|k}(x)$ is non negative, i.e., $$\mathfrak{F}^{i^{''}}_{j|k} (x) \geq 0.$$
where 
\begin{align}\label{eq:NDtheDeriv}
\mathfrak{F}^{i^{''}}_{j|k}(x) &= \frac{2}{x^3} -  \sum_{r \in \mathcal{N}_i \setminus \{k,j\}} \frac{2}{(x+P_{r,i})^3} -...\\  & \nonumber +  (-1)^{N_i} \frac{2}{(x+\sum_{r \in \mathcal{N}_i \setminus \{j,k\}}P_{r,i})^3}
\end{align}
To show that this is non-negative over the range of $x \in [0,1]$, we utilize the similarity between each term in (\ref{eq:NDtheDeriv}) and some properties of Unilateral Laplace Transform \cite{Oppenheim:1996}. Let the Laplace transform of a function $\mathfrak{g}(x)$ be given by:
\begin{align*}
\mathfrak{G}(s) = \int_{0}^{\infty} \mathfrak{g}(t) e^{-ts} {\rm dt}
\end{align*}
As is well-known, for $a \in [0,1]$ and region of convergence $s > -a$, $\mathcal{L} \{ t^2 e^{-at}\} = \frac{1}{(s+a)^3}$. 
Given the linearity of the Laplace transform, by replacing $s$ with $x$, and $a$ with appropriate values in each term in (\ref{eq:NDtheDeriv}), we find
\begin{align*}
\mathcal{L}^{-1}\{\mathfrak{F}^{i^{''}}_{j|k}(x)\} = t^2 e^{-x t} ~Z_{jk}
\end{align*} 
where 
\begin{align}\
Z_{jk}& =   1- \sum_{r\neq k, r \in\mathcal{N}_i} e^{- P_{r,i} t} +... \\ & \nonumber +  (-1)^{N_i} e^{- (\sum_{ r \in \mathcal{N}_i \setminus \{k,j\}} P_{r,i})t}
\end{align}
Clearly, to show that $\mathfrak{F}^{i^{''}}_{j|k}(x) \geq 0$, it is sufficient to show that $Z_{jk}$ is non negative.
\\ \\
\textbf{Claim :} $Z_{jk} \geq 0$.
\begin{proof}
Let us substitute $Z_{jk} = 1 - V_{jk} $ where
\begin{align} \label{eq:NDDefV}
V_{jk}  &= \sum_{ r \in \mathcal{N}_i \setminus \{k,j\}} v_r - \\ & \nonumber \sum_{ r\neq s, r,s  \in \mathcal{N}_i \setminus \{k,j\}} v_r \times v_s +...+ (-1)^{N_i-1} \prod_{r \in \mathcal{N}_i \setminus \{k,j\}} v_r,
\end{align}
where $v_r =e^{- P_{r,i} t} \in [0,1]$.
The reader might notice the similarity between (\ref{eq:NDDefV}) and the inclusion exclusion principle of $n=N_i-2$ independent events. To see this clearly, define the set of independent events $A_r$, $r = \{1,...,n\}$, each occurs with probability $v_r$. Then the probability of the union of these events is
\begin{align*}
\mathbb{P}(A_1 \cup A_2 \cup ... \cup A_n) &= \sum_{r}  \mathbb{P}(A_r) - \sum_{r\neq s} \mathbb{P}(A_r \cap A_s)  +... \\ & \nonumber +(-1)^{n+1}\mathbb{P}(A_1 \cup A_2 \cup ... \cup A_n) 
\\ \nonumber & = \sum_{r}  v_r - \sum_{r\neq s} v_r \times v_s +...\\ & \nonumber + (-1)^{n+1} \prod v_r
\end{align*}
By axioms of probability, the left hand side is $\leq 1$, i.e., $V_{jk}\leq 1$, and thus $Z_{jk} \geq 0$. 
\end{proof}
This concludes the proof of lemma \ref{Lm:NDFconvex}. 
\end{proof}
Since $\mathfrak{F^i}_{j|k}(x)$ is convex, we can write, \cite{boyd2004convex},
\begin{align} \label{eq:NDJensF}
\mathfrak{F}^i_{k|j}\Bigg( \frac{P_{j,i}+P_{k,i}}{2}\Bigg) & +  \mathfrak{F}^i_{j|k}\Bigg( \frac{P_{j,i}+P_{k,i}}{2}\Bigg) \\ & \nonumber =  2\mathfrak{F}^i_{k|j} \Bigg( \frac{P_{j,i}+P_{k,i}}{2}\Bigg) \\
&\leq \mathfrak{F}^i_{j|k}(P_{j,i}) + \mathfrak{F}^i_{k|j}(P_{k,i})
\end{align}
Then adding the terms $\mathfrak{I}^i_{j,k} +  \mathfrak{C}^i_{k,j}$ to the both sides of the inequality completes the proof.
\end{proof}

\subsection{Proof of Theorem \ref{Th:NDAvgOfTwo}} \label{App:NDAppendixE3}
  \begin{proof} Let $\mathbf{p} = [p_1, p_2 ..., p_n]^T$; we start with the case that $n$ is odd number. Then define ${\pmb \omega}_u$ an $n\times n$ matrix that has all one diagonal elements except a $2 \times 2$ block that starts at $(2u-1,2u-1)$ and equals to $\frac{1}{2} \mathbf{U}_{2\times2}$, where $\mathbf{U}_{u \times v}$ is all one matrix of size $u \times v$. We also define $\mathbf{O}_{u \times v}$ as all zero $u \times v$ matrix and $\mathbf{I}_{u \times u}$ as an identity matrix of size $u \times u$. For instance we have
\begin{align*}
{\pmb \omega}_2 = 
\begin{bmatrix}
\begin{matrix}
\begin{matrix}
\mathbf{I}_{2 \times 2} &  \mathbf{O}_{2\times 2} \\
\mathbf{O}_{2\times 2}   & \frac{1}{2} \mathbf{U}_{2\times 2} 
\end{matrix}  & {\bf O}_{4\times n-4} \\     
     \mathbf{O}_{n-4\times 4}    &  \mathbf{I}_{n-4\times n-4} \end{matrix}
\end{bmatrix}
\end{align*} 
 We also define $\widetilde{{\pmb \omega}}_u$, that is an all one diagonal matrix except the $2 \times 2$ block $\frac{1}{2}\mathbf{U}_{2 \times 2}$ that starts instead at element $(2u,2u)$.

Note that the process of replacing the first two elements of $\mathbf{p}$ with their average is equivalent to ${\pmb \omega}_1 \mathbf{p}$. Next define ${\bf \Omega} = \prod_{i}^{\frac{n-1}{2}} {\pmb \omega}_i$ and similarly we have $\bf \widetilde{\Omega} = \prod_{i}^{\frac{n-1}{2}}  \widetilde{{\pmb \omega}}_i$. It is easy to see that $ \mathbf{ \Omega p}$ replaces the pair of elements $(p_i,p_{i+1})$ for  $i=\{1,3,...,n-2\}$ of $\mathbf{p}$ with their averages, similarly $\mathbf{ \widetilde{\Omega} p}$ for  $(p_i,p_{i+1})$ $i=\{2,4,...,n-1\}$. Next, define:
$$ \mathbf{p}_{u} = (\mathbf{\widetilde{\Omega} \Omega)}^u \mathbf{p}$$
Then an algorithm described in Theorem \ref{Th:NDAvgOfTwo} can be the one in the following lemma:\\ 
\begin{lem}\label{lem:NDconvg}
 $\lim_{u \rightarrow \infty}  \mathbf{p}_{u} = \bar{\mathbf{p}}$
\end{lem}
\begin{proof}To prove the lemma, we first need to study the structure of the matrices $\bf \Omega$ and $\bf \widetilde{\Omega}$. By multiplying all ${\pmb \omega}_i$ we can see that ${\bf \Omega}$ is a block diagonal with each block equal to $\frac{1}{2} \mathbf{U}_{2\times 2}$ except the last $1 \times 1$ block is 1, i.e., we have 
\begin{align*}
{\bf \Omega} = 
\begin{bmatrix}
\begin{matrix}
    \frac{1}{2} \mathbf{U}_{2\times 2}  & \mathbf{O}_{2\times 2}            & \mathbf{O}_{2\times 2}   & \dots & \\ 
    \mathbf{O}_{2\times 2}              & \frac{1}{2} \mathbf{U}_{2\times 2} &  \mathbf{O}_{2\times 2} & \dots & \\
    \vdots & & \ddots & & \\
\end{matrix} &  \mathbf{O}_{n-1\times 1} \\
    \mathbf{O}_{1\times n-1} &  1
\end{bmatrix}
\end{align*}  
Similarly, $\bf \widetilde{\Omega}$ is a block diagonal, but the first $1\times 1$ block is equal to one, and the rest are $\frac{1}{2} \mathbf{U}_{2,2}$. Define $\mathbf{W} \triangleq \bf \widetilde{\Omega} \Omega$, then for $n\geq 5$ let $\mathbf{M}_{u}$ be an $n \times 2 $ matrix of all zero except an all one $4\times 2$ sub matrix, the first element of the sub matrix is at $(u+1,1)$, i.e., for instance 
\begin{align*}
\mathbf{M}_{1} =  \begin{bmatrix} \mathbf{O}_{1 \times 2}  \\ \mathbf{U}_{4 \times 2}   \\  \mathbf{O}_{n-5 \times 2}  \\ 
 \end{bmatrix}
\end{align*}
Then, it is easy to verify that $\mathbf{W}$ has the following form 
\begin{align*}
\mathbf{W} = 
\frac{1}{4} 
\begin{bmatrix}
\begin{matrix}
2\mathbf{U}_{1 \times 2} \\ \mathbf{U}_{2 \times 2} \\ \mathbf{O}_{n-3 \times 2}  \\  
\end{matrix}
 &    \mathbf{M}_{1}  & \mathbf{M}_{3}  & \hdots  & \mathbf{M}_{n-5} &    \begin{matrix}
\mathbf{O}_{n-2 \times 1}\\
 2\mathbf{U}_{2 \times 1} 
\end{matrix}   
 \end{bmatrix}
\end{align*}
 Note that every row and every column of $\mathbf{W}$ add to 1, this type of matrices are called doubly stochastic. To prove the lemma we utilize properties of such matrices. Remember that we need to show that 
$$ \mathbf{W}^{\infty} = \lim_{u\rightarrow \infty} \mathbf{W}^u = \frac{1}{n} \mathbf{U}_{n \times n}$$

Since $\mathbf{W}$ is a square matrix with non negative entries and every row adds to one, we can view it as a state transition matrix of a Markov Chain (MC) with $n$ states. The resultant MC is irreducible since any state $u$ can be accessible from any other state $v$. It is also aperiodic, since aperiodicity is a class property, it is enough to note that for this finite irreducible MC any state has a self transition, and thus period $d =1$. As in \cite{StochBookRoss}, the aforementioned properties indicate that there is a unique stationary distribution of the MC such that  
$$ \mathbf{W}^{\infty} = 
\begin{bmatrix}
w_1     &   w_2    &  \dots & w_n \\
w_1     &   w_2    &  \dots & w_n \\
\vdots  &   \vdots &        & \vdots \\
 w_1     &   w_2    &  \dots & w_n \\
 \end{bmatrix} $$
 Doubly stochastic matrices are closed under multiplication \cite{bhatia2013matrix}, i.e., $\mathbf{W}^u$ is a doubly stochastic matrix for any value $u$. Consequently, we have $ n w_i = 1 \implies w_i = \frac{1}{n}$, as a result we have 
 $$ \mathbf{W}^{\infty} = 
\frac{1}{n} \begin{bmatrix}
1     &   1    &  \dots & 1 \\
\vdots  &   \vdots &    \ddots    & \vdots \\
 1     &   1    &  \dots & 1 \\
 \end{bmatrix} = \frac{1}{n} \mathbf{U}_{n \times n}$$
Q.E.D.
\end{proof}
This also shows that averaging two elements at a time in the sequence indicated by Lemma \ref{lem:NDconvg} is one possible algorithm in Theorem \ref{Th:NDAvgOfTwo}. When $n$ is an even number, we need to slightly modify ${\pmb \omega}_u$ and $  \widetilde{{\pmb \omega}}_u$ such that we omit the last row and column in both matrices. The proof follows similar methodology, and is omitted for brevity.

\end{proof}
  
  \bibliographystyle{ieeetr}
\bibliography{RNDJournalPTh3}

\end{document}